\documentclass[review]{elsarticle}

\usepackage{lineno}
\usepackage{hyperref}
\usepackage{amssymb,amsmath,amsthm,mathtools}
\usepackage{changepage}
\usepackage[english]{babel}
\usepackage[T1]{fontenc}
\usepackage{microtype}
\usepackage[textsize=tiny]{todonotes}
\usepackage{tikz}
\usetikzlibrary{calc,math}
\usetikzlibrary{arrows, arrows.meta}
\usepackage{graphicx}
\usepackage{pgfplots}

\usepackage{enumitem}
\setlist{noitemsep,labelwidth=*,leftmargin=*,align=left}
\setlist[enumerate,1]{label=(\alph*)}
\setlist[description]{font=\normalfont,leftmargin=!}

\usepackage[ruled,noend]{algorithm2e}
\renewcommand{\SetProgSty}[1]{\renewcommand{\ProgSty}[1]{\textnormal{\csname#1\endcsname{##1}}\unskip}}
\SetProgSty{}
\SetFuncSty{textsc}
\SetFuncArgSty{}
\SetArgSty{}
\SetCommentSty{emph}
\DontPrintSemicolon

\SetKwComment{tcp}{\textup{\texttt{\#}} }{}
\SetKwComment{tcc}{''' }{ '''}
\SetKwProg{Def}{def}{:}{}

\newcommand{\algass}{\ensuremath{\leftarrow}}

\renewcommand{\O}[1]{\ensuremath{\mathcal{O}(#1)}}

\SetKw{Continue}{continue}
\SetKw{Break}{break}

\SetKwInput{Input}{Input}
\SetKwInput{Output}{Output}
\SetKw{Return}{\textbf{Return:}}
\SetKw{And}{\textbf{and}}
\SetKwComment{Comment}{/* }{ */}

\usepackage[capitalise]{cleveref}
     
\theoremstyle{definition}
\newtheorem{theorem}{Theorem}[section]

\newtheorem{corollary}[theorem]{Corollary}
\newtheorem{definition}[theorem]{Definition}

\newtheorem{lemma}[theorem]{Lemma}
\newtheorem{problem}[theorem]{Problem}

\newtheorem{observation}[theorem]{Observation}
\newtheorem{notation}[theorem]{Notation}
\newtheorem{assumption}[theorem]{Assumption}
\theoremstyle{definition}

\newcommand{\m}{m}
\newcommand{\M}{\ensuremath{\mathcal{M}}}
\newcommand{\Set}[1]{\left\{ #1 \right\}}
\DeclareMathOperator*{\F}{\mathcal{F}}

\pgfplotsset{compat=1.18}

\journal{Journal of ...}

\biboptions{authoryear}

\begin{document}
	
\begin{frontmatter}
		
    \title{Parametric Matroid Interdiction}

    %% Group authors per affiliation:
    \author[a]{Nils Hausbrandt}\corref{mycorrespondingauthor}
    \cortext[mycorrespondingauthor]{Corresponding author}
    \ead{nils.hausbrandt@math.rptu.de}
    \author[a]{Oliver Bachtler}
    \ead{o.bachtler@math.rptu.de}
    \author[a]{Stefan Ruzika}
    \ead{stefan.ruzika@math.rptu.de}
    \author[b]{Luca E. Sch\"afer}
    \ead{lucaelias.schaefer@aol.com}
    
    \address[a]{Department of Mathematics, University of Kaiserslautern-Landau, 67663 Kaiserslautern, Germany}
    
    \address[b]{Comma Soft AG, 53229 Bonn, Germany}
    
    \begin{abstract}
    We introduce the parametric matroid one-interdiction problem.
    Given a matroid, each element of its ground set is associated with a weight that depends linearly on a real parameter from a given parameter interval.
    The goal is to find, for each parameter value, one element that, when being removed, maximizes the weight of a minimum weight basis.
    The complexity of this problem can be measured by the number of slope changes of the piecewise linear function mapping the parameter to the weight of the optimal solution of the parametric matroid one-interdiction problem.
    We provide two polynomial upper bounds as well as a lower bound on the number of these slope changes.
    Using these, we develop algorithms that require a polynomial number of independence tests and analyse their running time in the special case of graphical matroids.
    \end{abstract}
    
    
    \end{frontmatter}
	
\section{Introduction}
Parametric matroid interdiction  combines three important areas of research in mathematical programming: matroid theory, parametric optimization, and the concept of interdiction in optimization problems. Before introducing parametric matroid interdiction problems formally, we briefly highlight each of these concepts and outline some key findings.

Matroids were first introduced by \cite{whitney1935} to generalize linear independence in linear algebra. Since independence is at the heart of many optimization problems, matroid theory is an important and well-studied area of research, cf.~\cite{wilson1973introduction,welsh2010matroid,Oxl11} with numerous applications, for example in combinatorial optimization, network theory and coding theory, cf.~\cite{white1992matroid,el2010index,kveton2014matroid,tamo2016optimal,ouyang2021importance}.

In parametric optimization, the objective function value of a feasible solution depends on the solution itself and, additionally, on a real-valued parameter.
Consequently, one is interested in finding an optimal solution to the problem for every possible parameter value.
A solution of the problem is given by a partition of the parameter interval into subintervals together with a feasible solution that is optimal for every parameter value in the corresponding subinterval.
The problem obtained by fixing the parameter to a fixed value is called non-parametric problem.
The piecewise linear function mapping the parameter to the optimal solution of the associated non-parametric problem is called optimal value function.
The points of slope change of the optimal value function are called breakpoints.
Thus, the number of breakpoints constitutes a natural complexity measure for parametric optimization problems.
Many parametric optimization problems are intractable, this means that their optimal value function can have super-polynomially many breakpoints, cf.~\cite{CarstensenNetwork, Ruhe} for the parametric minimum cost flow problem, cf.\ \cite{CarstensenParametric,Nikolova,Gajjar} for the parametric shortest path problem and \cite{Gassner} for the parametric assignment problem.
So far, research in parametric optimization focused on approaches that compute the optimal value function exactly, cf.\ \cite{eisner1976mathematical} and, only recently, approximation methods were considered. \cite{Ruzika} have constructed an FPTAS for a general class of parametric optimization problems (including the three problems mentioned above) in case that the non-parametric variant is solvable in polynomial time or admits an FPTAS.
In contrast to the intractable problems mentioned above, there are also parametric problems that admit polynomial time algorithms. Examples are the parametric maximum flow problem, cf.\ \cite{Arai, GGTAlgorithmus, Mccormick, {Scutella}} for different versions and the parametric matroid problem with its special case of parametric minimum spanning trees, see \cite{gusfield1979bound,katoh1983total,Agarwal,ParaMST}.

Interdiction problems investigate to which extent an optimization problem can be maximally interfered by a limited intervention.
These problems were first studied by \cite{harris1956fundamentals} and have become a highly topical area of research in recent years, cf.\ \cite{6b1f75254ade48c1a04886babfa6508a,SMITH2020797,NGUYEN2022239}.
An interdiction problem can be thought of as a game between two players in which the first player optimizes his objective function value and the other, called the interdictor, tries to disrupt the first player's objective as much as possible.
The operations of the interdictor have certain interdiction costs and are constrained by an interdiction budget.

Each of these three concepts---matroid theory, parametric optimization, and interdiction---has been well explored in various settings and a listing of the existing literature is beyond the scope of this article. There is also some work that considers some specific combination of these concepts.
\cite{linhares} have shown hardness results for matroid rank interdiction with an arbitrary non-negative interdiction budget.
We also refer to \cite{joret2015reducing} for the study of rank interdiction of matroids.
\cite{chestnut2017interdicting} have considered approximations for interdiction problems with \{0-1\}-objectives, including the matroid interdiction problem with submodular interdiction costs.
\cite{frederickson1998algorithms} have considered a related problem that considers the perturbability of matroids.

There is a large body of literature on minimum spanning tree interdiction. One of the first problems studied, in which the interdiction budget and the interdiction cost of each edge are equal to one, has been shown to be solvable in polynomial time, cf.\ \cite{HSU,IWANO,Bhattacharya,ShenHon}.
In the literature, the problem is often referred to as the most vital edge problem, since it asks for an edge (most vital edge) that, when removed, increases the cost of a minimum spanning tree as much as possible.
The $k$-most vital edge problem, where the interdiction budget is equal to a natural number $k >1$, has been proven to be $\mathsf{NP}$-hard, cf.\ \cite{lin1993most,FREDERICKSON1999244} even on complete graphs with binary edge weights, cf.\ \cite{Vanderpooten}.
For a fixed $k$, there are several algorithms that solve the $k$-most vital edge problem in polynomial time, cf.\ \cite{618087,LIANG19971889,LIANG2001319,BAZGAN20122888}.
\cite{GuoShrestha} have shown hardness results for the general minimum spanning tree interdiction problem.
We refer to \cite{Zenklusen} and \cite{linhares} for constant factor approximations that can be extended for the metric traveling salesman interdiction problem.
\cite{wei2021integer} have considered integer programming formulations.
\cite{NGUYEN2022239} have considered the minimum spanning tree perturbation problem, where the goal is to spend a fixed budget to increase the weight of the edges in order to increase the weight of the minimum spanning tree as much as possible.

In a parametric interdiction problem, the optimal objective function value of a feasible solution to the interdiction problem depends not only on the solution itself, but also on a real parameter from a given parameter set.
The goal is to find an optimal solution to the interdiction problem for each possible parameter value.
The function mapping the parameter to the corresponding optimal objective function value of the non-parametric interdiction problem is called optimal interdiction value function.
The points of slope change, called changepoints, provide a measure of complexity, since at such a point either the optimal interdiction strategy or the optimally interdicted solution changes.

Parametric interdiction problems open a large gap of research.
Although, there are a few interdiction problems that consider fuzzy arc lengths, cf.\ \cite{Fuzzy}, stochastic arc lengths, cf.\ \cite{ZHANG201862, Nguyen2022,punla2022shortest} or stochastic capacities, cf.\ \cite{Cormican}, the parametric shortest path interdiction problem is the only parametric interdiction problem investigated so far, cf.\ \cite{LinChern}.
They have constructed an algorithm whose running time is linear in the number of breakpoints of the parametric shortest path problem. Therefore, in general, the algorithm does not have a polynomial running time.
We also note that \cite{CHERN1995580} have performed a parametric analysis for two generalizations of the shortest path problem that can be applied to the interdiction of network flows.

There are at least three reasons that parametric interdiction is an exciting area of research.
First, since uncertainty is described by the parameter in the objective function, the problems can play an important role in robust optimization since an optimal solution provides a solution for infinitely many scenarios.
Second, if an optimal solution to the problem requires super-polynomially many interdiction strategies, no polynomial time algorithm can exist, even if $\mathsf{P}= \mathsf{NP}$, which strongly motivates the design of approximation algorithms.
Third, many interdiction problems are already $\mathsf{NP}$-hard, so again one cannot expect polynomial-time algorithms for their parametric version.

\paragraph{Our contribution}
In this article, we formally introduce the parametric matroid one-interdiction problem, which has not been investigated in the literature so far.
An element that, when removed from the matroid, increases the weight of a minimum weight basis as much as possible is called most vital element.
The goal is to compute a most vital element and the corresponding objective function value for each parameter value from a given parameter interval.
A point of slope change of the corresponding optimal interdiction value function mapping the parameter to the optimal objective function value of the matroid one-interdiction problem is called a changepoint.
We prove two upper bounds on the number of changepoints, namely that their number is in $\mathcal{O}(m^2 k^{\frac{1}{3}}\alpha(k))$ and $\mathcal{O}(mk^2)$ as well as a lower bound of $\Omega(m k^{\frac{1}{3}})$.
Here, $\alpha$ is a functional inverse of Ackermann's function, $m$ is the number of elements and $k$ is the rank of the matroid. 
We show some structural properties of the problem and, based on them, develop algorithms whose running time depends on different matroid operations that can be implemented in polynomial time for graphical matroids.

\paragraph{Outline}
In \cref{sec:prelim} we introduce the required preliminaries for this paper, including the notation regarding matroids and graphs, as well as the parametric matroid and minimum spanning tree one-interdiction problems.
Afterwards, we look at structural properties of the associated objective value functions of these problems and restrict the number of slope changes of those in \cref{sec:structural-properties}.
Finally, we use these structural properties to develop algorithms for both problems in \cref{sec:algorithms}.

\section{Preliminaries}\label{sec:prelim}

\textbf{Matroids.}
For a set $A$ and a singleton $\{b\}$, we use the notation $A+b$ or $A-b$ for $A\cup \{b\}$ or $A\setminus \{b\}$, respectively.
A tuple $\M = (E,\F)$ where $E$ is a finite set and $\emptyset\neq\F\subseteq 2^E$ is called \emph{matroid} if the following properties hold:
\begin{enumerate}
	\item The empty set $\emptyset$ is contained in $\F$.
	\item If $A\in\F$ and $B\subseteq A$, then also $B\in\F$.
	\item If $A,B \in\F$ and $\vert B\vert < \vert A\vert$, then there exists an element $a\in A\setminus B$ such that $B + a\in\F$.
\end{enumerate}
We denote the cardinality of $E$ by $\m$.
The subsets of $E$ in $\F$ are called \emph{independent sets of \M} and all other subsets of $E$ are called \emph{dependent}.
An inclusion-wise maximal independent set of \M{} is a \emph{basis of \M}, and  
the \emph{rank} $k\coloneqq rk(\M)$ of a matroid is the cardinality of a basis. A minimal dependent set is called a \emph{circuit}. We will need the following property of circuits later.
\begin{lemma}[{\cite[Proposition~1.4.12]{Oxl11}}]
	\label{strong-circuit-elimination-axiom}
	Let \M{} be a matroid, $C_1$, $C_2$ be circuits of \M, and $e$, $f\in E$ such that $e\in C_1\cap C_2$ and $f\in C_1\setminus C_2$.
	Then there exists a circuit $C$ such that $f\in C \subseteq (C_1\cup C_2)-e$.
\end{lemma}

From this property, it follows that the relation
\begin{displaymath}
	e \sim f :\iff e = f \text{ or } \M \text{ has a circuit containing $e$ and } f
\end{displaymath}
is an equivalence relation \cite[Proposition~4.1.2]{Oxl11}.
The equivalence classes with respect to $\sim$ are called the \emph{components} of $\M$.

We shall need to regard restrictions of matroids throughout this paper.
To this end, for a subset $E'\subseteq E$, we denote the matroid $(E',\F')$ where $\F'\coloneqq\left\{F\in\F\colon\, F\subseteq E' \right\}$ by $\M|E'$.
For $e\in E$, we write $\M_e$ for $\M|(E-e)$.

We assume that each element $e\in E$ has a parametric weight $w(e,\lambda)\coloneqq a_e + \lambda b_e$ depending linearly on a parameter $\lambda \in I\subseteq \mathbb{R}$, where $I$ is an interval and $a_e,b_e \in \mathbb{Q}$.
The weight of a basis $B$ is defined as $w(B,\lambda)\coloneqq\sum_{e\in B} w(e,\lambda)$.

For the rest of the paper, we fix the following notation.
\begin{notation}
	Let $\M = (E,\F)$ be a matroid with parametric weights $w(e,\lambda)$ and rank $k$.
\end{notation}
We recall that $\M$ is a graphic matroid if $E$ is the edge set of a graph $G$ and $\F$ contains the sets of edges $F$ such that the subgraph $G[F]$ induced by the edges in $F$ is acyclic.
We also note that the components of a graphic matroid are exactly the 2-connected components of the corresponding graph, cf.\ \cite[Proposition~4.1.7]{Oxl11}.

\textbf{The parametric matroid problem.} 
The goal of the parametric matroid problem is to compute, for each parameter value $\lambda \in I$, a minimum weight basis $B_\lambda^*$ with respect to the weights $w(e,\lambda)$.
The function $w\colon I \to \mathbb{R}, \lambda \mapsto w(B_\lambda^*,\lambda)$ mapping the parameter to the weight of a minimum weight basis is called optimal value function (of the parametric matroid problem). 
\cite{Gusfield} showed that $w$ is piecewise linear and concave. 
The points $\lambda\in I$ where the slope of $w$ changes are called breakpoints. 
It is well-known that a breakpoint can only occur at a so-called equality point, which is a point at which the weights of two elements of $\M$ become equal. 
By $\lambda(e,f)$ we denote the equality point where $w(e,\lambda) = w(f,\lambda)$.
If this point does not exist, then we set $\lambda(e,f)$ to $-\infty$.
Moreover, we write $\lambda(e\to f)$ for $\lambda(e,f)$ if $w(e,\lambda) < w(f,\lambda)$ for $\lambda < \lambda(e,f)$ (and thus $w(e,\lambda) > w(f,\lambda)$ for $\lambda > \lambda(e,f)$), meaning that $e$ was the element of lower weight before the equality point, and $f$ is the lighter one after.

There are at most $\binom{m}{2}\in\mathcal{O}(m^2)$ many equality points and $\Theta(\m k^{\frac{1}{3}})$ many breakpoints, cf.\ \cite{dey1998improved,eppstein1995geometric}.
A simple algorithm for the parametric matroid problem is to traverse the sorted list of equality points and check by an independence test whether an equality point is indeed a breakpoint. 
The running time is $\mathcal{O}(\m^2f(\m))$ where $f(\m)$ is the time needed to perform a single independence test.

\textbf{Parametric matroid interdiction.}
We now define a most vital element and use it to define the optimal interdiction value function.
\begin{definition}[Most vital element] \label{definition mve}
	Let $\lambda\in I$.
	For an element $e \in E$, we denote by $B_\lambda^e$ a minimum weight basis of $\M_e$.
	If $\M_e$ does not have a basis of rank $k$, we set $w(B_\lambda^e,\lambda)=\infty$ for all $\lambda\in I$.
	An element $e^*\in E$ is called \emph{most vital element} at $\lambda$ if $w(B_\lambda^{e^*},\lambda) \geq w(B_\lambda^e,\lambda)$ for all $e\in E$.
\end{definition}
\begin{definition}[Optimal interdiction value function] \label{definition opt interdiction value function}
	For $e\in E$, we define the function $y_e$ by $y_e\colon I\to\mathbb{R}$, $\lambda\mapsto w(B_\lambda^e,\lambda)$ mapping the parameter $\lambda$ to the weight of a minimum weight basis of $\M_e$ at $\lambda$. 
	We define $y(\lambda)\coloneqq\max\{y_e(\lambda)\colon\, e\in E\}$
	as the weight of an optimal one-interdicted matroid at $\lambda$.
	The function $y\colon I\to\mathbb{R}$, $\lambda\mapsto y(\lambda)$ is called the \emph{optimal interdiction value function}.
\end{definition}

With this notation, we can formulate the parametric matroid one-interdiction problem as follows.
\begin{problem}[Parametric matroid one-interdiction problem]\label{problem matroid}
	Given a matroid~$\M$ with parametric weights $w(e,\lambda)$ and the parameter interval~$I$, compute, for every $\lambda\in I$, a most vital element $e^*$ and the corresponding objective function value $y(\lambda) = y_{e^*}(\lambda)$.
\end{problem}
Our goal is to solve this problem efficiently, meaning we need to find the intervals on which $y$ is linear and the most vital elements do not change effectively.

In order to avoid certain special cases, we make the following assumptions, which are essentially without loss of generality.
\begin{assumption}\label{assumptions} In this paper, we make the following assumptions.
	\begin{enumerate}
		\item There exists a unique minimum weight basis $B_\lambda^*$ for all $\lambda\in I$.
		\item There exists a basis $B_\lambda^e$ with rank $k$ for every $e\in E$ and $\lambda\in I$.
		\item No two pairs of weights $w(e,\lambda)$ become equal simultaneously.
	\end{enumerate}
\end{assumption}
To see that these are, indeed, without loss of generality, note that ties can be solved by a fixed ordering of the elements in $E$ yielding a unique minimum weight basis.
If $\M_e$ does not contain a basis of rank $k$ for some $e\in E$, the problem can be trivially solved by always interdicting $e$, yielding a value of infinity at each point $\lambda\in I$.
An infinitesimally change in a weight function $w(e,\lambda)$ causes a minimal perturbation in the intercept with another function, and does not significantly change the optimal value function or any of the functions $y_e$, cf.\ \cite{ParaMST} for parametric minimum spanning trees.
Since $y$ is given by the upper envelope of the functions $y_e$, also $y$ is not changed significantly.

We note that, throughout the paper, we often treat the case that $I=\mathbb R$, for example when determining upper bounds on the number of breakpoints.
This, too, is without loss of generality, since smaller intervals have fewer breakpoints and the worst case is obtained for all of $\mathbb R$.

Since functions $y_e$ are piecewise linear and continuous, the same holds true for $y$.
\begin{observation}
	The optimal interdiction value function $y$, which is the upper envelope of the functions $y_e$, is piecewise linear and continuous. 
\end{observation}

Since we are interested in determining the number of linear pieces of $y$ and the points where the slopes change, the following definition is helpful.
\begin{definition}
	\label{def:change-break-interdiction-points}
	The points of slope change of $y$ are called \emph{changepoints}, which are partitioned into \emph{breakpoints} and \emph{interdiction points}. 
	A breakpoint $\lambda$ of $y$ arises from a breakpoint of a function $y_{e^*}$, where $e^*$ is a most vital element before and after $\lambda$. 
	An interdiction point $\lambda$ of $y$ occurs if the most vital element changes at $\lambda$, that is, if a function $y_{e^*}$ intersects a function $y_{f^*}$, where $e^*$ and $f^*$ are most vital elements before and after $\lambda$, respectively.
\end{definition}

\textbf{Minimum spanning tree interdiction.}
We end the preliminaries with some notes on a special case of the matroid interdiction problem.
If we assume that $\M$ is a graphic matroid, we obtain the minimum spanning tree interdiction problem, which we now introduce. 
To this end, let $G$ be an undirected and connected graph with node set $V$ and edge set $E$ with $\left| V \right| \coloneqq n \in \mathbb{N}$ and $\left| E \right| \coloneqq m \in \mathbb{N}$. 
Each edge $e\in E$ has a parametric weight $w(e,\lambda)\coloneqq a_e + \lambda b_e$ where $\lambda \in I$.
The rank $k$ corresponds to the number of edges of a spanning tree of $G$, which is $n-1$.
Let $\lambda\in I$.
The minimum weight basis $B_\lambda^*$ corresponds to the minimum spanning tree $T_\lambda^*$ of $G$ at $\lambda$ and $B_\lambda^e$ to the minimum spanning tree $T_\lambda^e$ of $G-e = (V,E-e)$ at $\lambda$.
Here, a most vital element $e^*\in E$ as defined in \cref{definition mve} is called most vital edge.
In complete analogy, we obtain the functions $y_e$ with $y_e(\lambda) = w(T_\lambda^e,\lambda)$ and the optimal interdiction value function $y$ as their upper envelope.
The resulting problem can be stated as follows.
\begin{problem}[Parametric minimum spanning tree one-interdiction problem]\label{problem mst}
	Given a graph $G$ with parametric weights $w(e,\lambda)$ and the parameter interval $I$, compute, for any $\lambda\in I$, a most vital edge and the corresponding objective function value $y(\lambda)$.
\end{problem}

\section{Structural properties} \label{sec:structural-properties}

In this section, we derive bounds on the number of changepoints of the optimal interdiction value function $y$.
A first upper bound follows directly from the theory of Davenport–Schinzel Sequences, cf.\ \cite{AGARWAL20001}.
\begin{theorem}\label{theorem upper bound 1}
	The optimal interdiction value function $y$ of the parametric matroid one-interdiction problem has at most $\mathcal{O}(\m^2 k^{\frac{1}{3}}\alpha(\m))$ many changepoints.
\end{theorem}
\begin{proof}
	\cite[Corollary~2.18]{AGARWAL20001} showed that the upper envelope of $\m$ piecewise linear, real-valued, and continuous functions has at most $\mathcal{O}(t\alpha(\m))$ many changepoints where $t$ is the total number of linear pieces in the graphs of these functions.
	For our problem, $t$ is in $\mathcal{O}(\m^2 k^{\frac{1}{3}})$ since each of the $\m$ functions $y_e$ has at most $\mathcal{O}(\m k^{\frac{1}{3}})$ many breakpoints, cf.\ \cite{dey1998improved}.
\end{proof}

Aided by the following lemma, we will be able to reduce the number of functions $y_e$ needed for the computation of $y$ from $m$ to $k$.
\begin{lemma}\label{lemma only elements of B*}
	For any $\lambda\in I$, a most vital element is an element of $B_\lambda^*$.
\end{lemma}
\begin{proof}
	Let $e\notin B_\lambda^*$, then $B_\lambda^*$ is a basis of $\M_{e}$, giving us that $B_\lambda^e = B_\lambda^*$.
	For every $e\in B_\lambda^*$, we get that $B_\lambda^e \neq B_\lambda^*$ and $w(B_\lambda^e,\lambda)\geq w(B_\lambda^*,\lambda)$.
	In case of equality, by our uniqueness assumption, $B_\lambda^*$ wins ties.
	Therefore, no element in $E\setminus B_\lambda^*$ can be most vital.
\end{proof}

The following definition and lemma generalize the definition of replacement edges for graphic matroids from \cite{HSU} to general matroids with parametric weights.
\begin{definition}[Replacement element]
	Let $\lambda \in I$. 
	For $e\in B_\lambda^*$, we define $R_\lambda(e)$ as the set of all replacement candidates of $e$ in the minimum weight basis at~$\lambda$, i.~e.\ $R_\lambda(e)\coloneqq\{r\in E \setminus B_\lambda^* \;\colon\; B_\lambda^*-e+r \in \F\}$. 
	We define the replacement element of $e$ at $\lambda$ as $r_\lambda(e) = \arg\min\Set{w(r,\lambda)\colon\, r\in R_\lambda(e)}$.  
\end{definition}
Note that by sorting the elements in a predetermined way, we can assume without loss of generality that $r_\lambda(e)$ is unique.
\begin{lemma} \label{lemma replacement element}
	Let $\lambda \in I$ and $e\in B_\lambda^*$. 
	Then, $B_\lambda^e = B_\lambda^* -e + r_\lambda(e)$.
\end{lemma}
\begin{proof}
	Let $\lambda\in I$. 
	Let $w(e_1,\lambda)~\leq~\dotsc~\leq~w(e_m,\lambda)$ and $e_i=e$.
	We apply the greedy algorithm on $\M$ and on $\M_e$.
	Without loss of generality, we can assume that the greedy algorithm considers the elements in $E$ and $E-e$ in the same order and computes $B_\lambda^*$ in $\M$.
	Let $F_l$ and $F_l'$ be the independent sets obtained after iteration~$l$ on $E$ and $E-e$, respectively.
	(To keep the iteration indices consistent, we assume the greedy algorithm does nothing in iteration $i$ on $\M_e$.)
	The algorithm initializes $F_0 = F_0' =\emptyset$ with the empty set and tests in iteration~$l$ if $F_{l-1} + e_l \in \F$ and $F_{l-1}' +e_l \in \F$.
	Let $e_j=r_\lambda(e)$.
	Then, $F_l = F_l'$ for $l<i$ and $F_i' = F_i -e_i$.
	
	Inductively, we now show that the equality $F_l' = F_l -e_i$, which holds at $l=i$, remains true for $i<l<j$.
	To this end, we assume that the equality holds for $l-1$ and show that $F_{l-1} + e_l \in \F$ if and only if $F_{l-1}' +e_l\in\F$:
	if $F_{l-1} + e_l \in \F$, then also $F_{l-1}' +e_l \subseteq F_{l-1} + e_l \in \F$. 
	Let $F_{l-1}' +e_l \in \F$ and suppose that $F_{l-1} +e_l \notin \F$. 
	Complete the set $F_{l-1}' +e_l$ with elements of $B_\lambda^*$ to a basis $\hat{B} \subseteq B_\lambda^* - e_i + e_l \in \F$. 
	Hence, it follows that $\hat{B} = B_\lambda^* - e_i + e_l \in \F$ and $e_l \in R_\lambda(e_i)$.
	But since $w(e_l,\lambda) < w(e_j,\lambda)$, this is a contradiction.
	
	Then, in iteration $l=j$, it follows that $F_{j-1}+e_j \notin\F$ as $e_j \notin B_\lambda^*$ and $F_{j-1}'+e_j = F_{j-1} - e_i +e_j \subseteq B_\lambda^* - e_i +e_j \in\F$ as $e_j \in R_\lambda(e)$.
	Thus $F_j' = F_j - e_i + e_j$.
	
	For $l>j$ we show, inductively, that the equality $F_l' = F_l -e_i + e_j$, which holds at $l=j$, remains true Let $l>j$.
	Again, we assume it holds for $l-1$ and show that $F_{l-1} + e_l \in \F$ if and only if $F_{l-1}' +e_l\in\F$.
	If $F_{l-1}+e_l \in\F$, then $\lvert F_{l-1}'\rvert = \lvert F_{l-1}+e_l\rvert -1$ and hence  there exists a $f\in F_{l-1}+e_l = F_{l-1}' -e_i +e_j +e_l$ such that $F_{l-1}' +f\in\F$. Since $F_{i-1}'+e_i\notin\F$, it follows that $e_l\neq e_i$ and hence $f=e_l$. 
	If $F_{l-1}'+e_l \in\F$, it follows analogously that $F_{l-1}+e_l \in\F$.
	This shows that $B_\lambda^e = F_k' = F_k - e_i + e_j = B_\lambda^* - e_i + e_j$.
\end{proof}

Note that, since the objective function is linear between equality points, the optimal bases are unaffected in these intervals.
\begin{observation}\label{observation replacement element}
	Between two consecutive equality points, all replacement candidates and all replacement elements remain unchanged. 
\end{observation}

\begin{lemma}\label{lemma r(e_j) = e}
	Let $\lambda^*=\lambda(e\to f)$ be a breakpoint of $w$ and $\lambda_l$ and $\lambda_r$ be the next smaller and larger equality point of $\lambda^*$, respectively. Let $B$ be the minimum weight basis for $\lambda\in(\lambda_l,\lambda^*]$ and $B'=B-e+f$ be the minimum weight basis for $\lambda\in(\lambda^* ,\lambda_r]$.
	Then, for $\lambda \in (\lambda_l ,\lambda^*]$, it holds that $r_\lambda(e) = f$ and for $\lambda \in (\lambda^* ,\lambda_r]$ it holds that $r_\lambda(f) = e$.
\end{lemma}
\begin{proof}
	We show that $r_\lambda(f) = e$ for $\lambda \in (\lambda^* ,\lambda_r]$, the proof that $r_\lambda(e) = f$ for $\lambda \in (\lambda_l ,\lambda^*]$ works analogously.
	According to \cref{observation replacement element}, it suffices to consider a fixed $\lambda \in (\lambda^* ,\lambda_r]$.
	Note that $e \in R_\lambda(f)$, since $B'-f+e = B \in \F$.
	Suppose that there exists an element $g \in R_\lambda(f)$ and $w(g,\lambda) < w(e,\lambda)$. Then, it holds that $w(g,\lambda) < w(e,\lambda)$ for all $\lambda \in (\lambda_l,\lambda_r]$.
	Before testing the elements $e$ and $f$, the greedy algorithm tests the element $g$ in some iteration $p$ and rejects it, since $g$ is not an element of $B$ or of $B'$.
	Let $B_p$ be the independent set computed by the greedy algorithm before iteration $p$.
	Then, $B_p + g \notin \F$ which is a contradiction since $B_p + g \subseteq B'-f+g \in \F$.
\end{proof}
\begin{corollary}\label{corollary y_e(lambda^*)=y_f(lambda^*)}
	Let $\lambda^*=\lambda(e\to f)$ be a breakpoint of $w$ and $\lambda_l$ and $\lambda_r$ be the next smaller and larger equality point of $\lambda^*$, respectively. 
	Let $B$ be the minimum weight basis for $\lambda\in(\lambda_l,\lambda^*]$ and $B'=B-e+f$ be the minimum weight basis for $\lambda\in(\lambda^* ,\lambda_r]$.
	Then, it holds that $y_e(\lambda^*)=y_f(\lambda^*)$.    
\end{corollary}
\begin{proof}
	For $\lambda \in (\lambda_l,\lambda^*]$, it holds that $y_e(\lambda) = w(B,\lambda) - w(e,\lambda) + w(r_\lambda(e),\lambda)$ and, for $\lambda \in (\lambda^* ,\lambda_r]$, it holds that $y_f(\lambda) = w(B',\lambda) - w(f,\lambda) + w(r_\lambda(f),\lambda)$ by \cref{lemma replacement element} and \cref{observation replacement element}.
	Furthermore, it holds that $w(e,\lambda^*)=w(f,\lambda^*)$ and $w(B,\lambda^*)=w(B',\lambda^*)$.
	Thus, the claim follows from \cref{lemma r(e_j) = e}.
\end{proof}

We can now reduce the number of functions for computing $y$ from $m$ to $k$. When an element $e$ leaves the optimal basis $B_\lambda^*$ and gets replaced by an element $f$ at a point $\lambda(e\to f)$, we no longer need to consider $y_e$ (until $e$ possibly re-enters $B_\lambda^*$). Instead, we can use the function $y_f$ after the point $\lambda(e\to f)$.
\begin{theorem}\label{theorem upper bound 2}
	The optimal interdiction value function $y$ of the parametric matroid one-interdiction problem has at most $\mathcal{O}(\m^2 k^{\frac{1}{3}}\alpha(k))$ many changepoints.
\end{theorem}
\begin{proof}
	According to \cref{lemma only elements of B*}, before the first breakpoint of $w$, only the $k$ functions $y_e$ with $e\in B_\lambda^*$ are relevant for the computation of $y$. Now, let $\lambda^*=\lambda(e\to f)$ be a breakpoint of $w$ where element $e$ leaves $B_\lambda^*$ and gets replaced by $f$. Let $\lambda_l$ and $\lambda_r$ be the next smaller and larger equality point of $\lambda^*$, respectively. Then, it is sufficient to consider the functions $y_g$ with $g\in B_\lambda^* - e$ and the function given by
	\begin{displaymath}
		\lambda \mapsto
		\begin{cases}
			y_e(\lambda),  & \text{for} \, \lambda \in (\lambda_l,\lambda^*] \\
			y_f(\lambda), & \text{for} \, \lambda \in (\lambda^*, \lambda_r]
		\end{cases}
	\end{displaymath}
	on the interval $(\lambda_l,\lambda_r]$ by \cref{lemma only elements of B*}.
	The latter function is continuous on $(\lambda_l,\lambda_r]$ by  \cref{corollary y_e(lambda^*)=y_f(lambda^*)}.
	Thus, inductively, over the entire parameter interval~$I$, it is sufficient to consider $k$ continuous and piecewise linear functions for the computation of $y$.
	The claim follows from the upper bound of \cite{AGARWAL20001} as in \cref{theorem upper bound 1}.
\end{proof}

We note that there are, in fact, only $\binom{m}{2}\in\mathcal{O}(m^2)$ many equality points at which the functions $y_e$ can break.
So, not all of the $\mathcal{O}(\m^2 k^{\frac{1}{3}})$ many breakpoints counted above can be distinct.
Actually, we can obtain a tighter bound, namely that only $\mathcal{O}(mk)$ many of these points are candidates for breakpoints.
\begin{theorem}\label{theorem mk breakpoints}
	There are at most $2km$ equality points where one of the functions~$y_e$, for $e\in E$, or the optimal value function $w$ has a breakpoint.
\end{theorem}
\begin{proof}
	We need to introduce a bit of notation for this proof.
	We set
	\begin{align*}
		E_e^<(\lambda) \coloneqq \Set{f\in E\colon\, w(f,\lambda) < w(e,\lambda)}\qquad & \M_e^<(\lambda) \coloneqq \M|E_e^<(\lambda), \\
		E_e(\lambda) \coloneqq \Set{f\in E_e^<(\lambda) \colon\, \lambda(e,f) < \lambda}\qquad & \M_e(\lambda) \coloneqq \M|E_e(\lambda).
	\end{align*}
	Note that $E_e(\lambda) \subseteq E_e(\lambda')$ for all $e\in E$ and $\lambda \leq \lambda'$.
	
	With this notation, let $e\in E$ be fixed and $\lambda_1\coloneqq \lambda(e\to f_1) < \ldots < \lambda_p\coloneqq \lambda(e\to f_p)$ be all equality points of the form $\lambda(e\to f)$ in ascending order.
	We show that there are at most $2k$ points amongst these where $w$ or one of the functions $y_g$ breaks, for $g\in E$.
	Since all equality points are of this form, for some $e\in E$, this implies the claim.
	
	First, we observe that the set $E_e(\lambda)$ remains unchanged on the intervals $(\lambda_p,\infty)$ and $(\lambda_{i-1},\lambda_i]$, where $i\in\Set{1,\ldots,p}$ and $\lambda_0=-\infty$.
	We now regard an equality point $\lambda(e\to f_i)$ where $y_g$ breaks.
	For convenience, we set $f\coloneqq f_i$ and $\lambda\coloneqq \lambda_i = \lambda(e\to f)$.
	Let $B^g$ be an optimal basis before and at $\lambda$ in $\M_g$.
	Since $y_g$ has a breakpoint at $\lambda$, $B^g - e + f$ is an optimal basis after $\lambda$.
	Additionally, let $\lambda'\in(\lambda_i,\lambda_{i+1}]$.
	Then $E_e(\lambda') = E_e(\lambda)+f$.
	
	\textbf{Claim:} Let $F\subseteq E_e^<(\lambda)-g$ be independent in $\M_e^<(\lambda)-g$.
	Then $F\cup\{f\}$ is independent in $\M_e^<(\lambda')-g$.
	
	\begin{proof}
		Let $F\subseteq E_e^<(\lambda)-g$ be independent and suppose $F + f$ is not.
		Augment $F$ to a basis $B'$ of $\M_e^<(\lambda)-g$ and let $B''\coloneqq B^g\cap E_e^<(\lambda)$.
		Now let $\lambda^+ > \lambda$ be chosen such that $\lambda^+$ is smaller than any equality point that is greater than $\lambda$.
		In particular, the greedy algorithm yields $B^g - e + f$ at $\lambda^+$ and both $B'$ and $B''$ are independent in $\M_e^<(\lambda^+)-g = \M_e^<(\lambda) + f - g$.
		
		Since $B''+f \subseteq B^g - e + f$ is independent and $B'+f \supseteq F + f$ is not, we get that $B'$ is a basis of $\M_e^<(\lambda^+)-g$ and $|B''| < |B'|$.
		Hence, there exists an $h\in B'\setminus B''$ such that $B''+ h$ is independent.
		But, since $w(h,\lambda^+) < w(f,\lambda^+)$ and the greedy algorithm rejected $h$ when in possession of a subset $\tilde{B}\subseteq B''$, we must have that $B''+h$ is dependent, which is a contradiction.
	\end{proof}
	
	From the claim, we can immediately conclude that $rk(\M_e(\lambda)-g) < rk(\M_e(\lambda')-g)$, since we can add $f$ to a basis.
	For a breakpoint of $w$ we get the analogous result, the difference being that the claim then holds for independent subsets of $E_e^<(\lambda)$ and we get a rank increase for $\M_e(\lambda)$.
	
	We now distinguish two cases for the rank of $\M_e(\lambda') = \M_e(\lambda) + f$:
	\begin{enumerate}
		\item $rk(\M_e(\lambda')) > rk(\M_e(\lambda))$, that is, $rk(\M_e(\lambda')) = rk(\M_e(\lambda)) + 1$ or \label{itm:matroid-restr-rank-increase}
		\item $rk(\M_e(\lambda')) = rk(\M_e(\lambda))$.   \label{itm:matroid-restr-same-rank}     
	\end{enumerate}
	
	Case~\ref{itm:matroid-restr-rank-increase} always occurs if we are at a breakpoint of $w$.
	It can occur at most $k$ times before the rank of the entire matroid $\M$ is reached.
	To see that this is also true for Case~\ref{itm:matroid-restr-same-rank}, we observe that there must be a circuit $C_1$ in $\M_e(\lambda')$ that contains both $f$ and $g$:
	let $B$ be a basis of $\M_1 \coloneqq \M_e(\lambda)-g$.
	Since $rk(\M_1 + f + g) \geq rk(\M_1+f) = rk(\M_1) + 1$ by the claim and $rk(\M_1 + g + f) = rk(\M_1 + g) \leq rk(\M_1) + 1$ by Case~\ref{itm:matroid-restr-same-rank}, we get that $rk(\M_1 + g) = rk(\M_1) + 1$.
	From this, we can deduce that $B + f$ and $B + g$ are independent, but $B+f+g$ is not.
	Therefore, $\M_1 + f + g$ must have a circuit containing both $f$ and $g$.        
	
	Now, suppose that $g$ is contained in a circuit $C_2$ of $\M_e(\lambda)$.
	By Lemma~\ref{strong-circuit-elimination-axiom}, there exists a circuit $C$ of $\M_e(\lambda')$ that contains $f$ but not $g$.
	But then, $C-f$ is an independent set of $\M_e(\lambda)$ to which we cannot add $f$, contradicting the claim.
	
	To get a bound on the occurrences of Case~\ref{itm:matroid-restr-same-rank}, we note that it can occur at most once for any element of $E$, since when it occurs at an element~$g$, this element is not part of a circuit, and this is no longer true afterwards.
	On the other hand, if Case~\ref{itm:matroid-restr-same-rank} occurred for elements $g_1,\ldots,g_n$, then $\Set{g_1,\ldots,g_n}$ is independent in $\M$.
	This is true for $n=1$ since $g_1$ is not part of a circuit.
	Inductively, we get an independent set $\Set{g_1,\ldots, g_{n-1}}$ and an element $g_n$ that is not part of a circuit, yielding that $\Set{g_1,\ldots,g_n}$ is independent.
	So, once more, the occurrences are bounded by the rank of $\M$.
\end{proof}

From the theorem above, it directly follows that there are only $\mathcal{O}(mk)$ many breakpoints of $y$.
\begin{theorem}\label{theorem mk^2 changepoints}
	There are at most $\mathcal{O}(mk^2)$ many changepoints of $y$.
\end{theorem}
\begin{proof}
	By \cref{theorem mk breakpoints}, there are at most $\mathcal{O}(mk)$ many breakpoints of $y$ and between two successive breakpoints all functions $y_e$ are linear. According to \cref{lemma only elements of B*} only the $k$ functions $y_e$ where $e \in B_\lambda^*$ are relevant for the computation of $y$ and can produce at most $k-1$ interdiction points between to successive breakpoints. 
	This yields at most $\mathcal{O}(mk^2)$ many interdiction points and therefore also at most $\mathcal{O}(mk^2)$ many changepoints.
\end{proof}

Finally, we provide a lower bound on the number of changepoints of $y$ by transferring the bound from the non-interdiction variant.
\begin{theorem}\label{theorem lower bound}
	The optimal interdiction value function $y$ of the parametric matroid one-interdiction problem has at least $\Omega(\m k^{\frac{1}{3}})$ many changepoints.
\end{theorem}
\begin{proof}
	Let $(E,\F)$ be a matroid with parametric weights $w(e,\lambda)$. We define $(E',\F')$ by $E' = E \cup \{e'\colon e\in E\}$, $w(e',\lambda)=w(e,\lambda)$ for all $\lambda\in I, e\in E$ and $\F'=\Set{G'\cup F\colon\, G \mathop{\dot{\cup}} F\in\F}$ where $F'=\Set{e'\colon\, e \in F}$ for $F\in \F$.
	Then, $(E',\F')$ is again a matroid.
	Let $w$ be the optimal value function of $(E,\F)$. Let $w'$ be the optimal value function, $y'$ be the optimal interdiction value function and $B'_\lambda$ be an optimal basis of $(E',\F')$.
	Let $e \in B'_\lambda$. We assume without loss of generality that $e\in E$. Then, 
	$B_\lambda' - e + e'$ is a minimum weight basis of $(E',\F')$ at $\lambda$ and hence $y'\equiv y'_e \equiv w$.
	The claim follows by taking as $(E,\F)$ the matroid whose optimal value function has at least $\Omega(\m k^{\frac{1}{3}})$ many breakpoints, cf.\ \cite{eppstein1995geometric}.
\end{proof}

\section{Algorithmic consequences} \label{sec:algorithms}

Here, we develop algorithms for computing the optimal interdiction value function $y$ and for finding the solutions corresponding to the linear segments of $y$.
In the running time analyses, we shall need a few operations that we now provide notation for.
We will flesh out the details of how these can be implemented when looking at the special case of graphic matroids at the end of this section.
\begin{definition}[Operations on matroids.]
	We use the following notation to denote running times of certain matroid operations:
	\begin{description}
		\item[$f(\m)$] the time needed to perform a single independence test.
		\item[$g(\m)$] the time needed to update the components of $\M|E'$ to $\M|(E'+f)$.
		\item[$h(\m)$] the time needed to compute a replacement element $r_\lambda(e)$ for a given $\lambda\in I$ and $e\in B_\lambda^*$.
	\end{description}
\end{definition}
It will also be helpful to have amortised versions of the above operations since they allow for better running times when we apply them to graphic matroids.
To this end, we let $G(m)$ denote the time needed to do $m$ component updates, for the matroids $\M|E_0,\ldots,\M|E_m$ where $E_i\supseteq E_{i-1}$ and $E_i\setminus E_{i-1}$ contains exactly one element.
In particular, $E_0=\emptyset$ and $E_m = E$.
Finally, we let $H(m)$ denote the amortised cost for computing all $k$ replacement elements for elements $e\in B_\lambda^*$.
Since $G(m)\in \O{mg(m)}$ and $H(m)\in \O{kh(m)}$, running times without the amortised functions are easily obtained.

\textbf{Algorithms for the parametric matroid interdiction problem.}
We start by algorithmically using \cref{theorem upper bound 1}.
\begin{theorem}\label{theorem upper bound 1 computation}
	\cref{problem matroid} can be solved in time
	\begin{displaymath}
		\O{m^2(kf(m)+k^{\frac{1}{3}}\log(mk))}. 
	\end{displaymath}
\end{theorem}
\begin{proof}
	We compute $y$ as the upper envelope of the functions $y_e$ for $e\in E$.
	To this end, we need to determine the functions $y_e$ first. 
	We compute all equality points and sort them in $\mathcal{O}(m^2\log m)$ time.
	
	Next, we use the greedy algorithm to find an optimal basis $B$ for $\M$ and optimal bases $B^e$ for $\M_e$, $e\in E$ for some value $\lambda$ before the first equality point.
	We then iterate over the equality points in ascending order.
	Regarding an equality point $\lambda = \lambda(e\to f)$.
	If a function $y_g$ breaks at $\lambda$, then the basis $B^g$ is replaced by $B^g - e + f$.
	Thus, we simply need to check whether $B^g - e + f$ is independent for all bases $B^g$ with $e\in B^g$ and $f\notin B^g$.
	But, since all bases $B^g$ with $g\notin B$ are identical, we only need to make $k+1$ independence tests to answer these questions.
	Thus, the entire procedure requires $\mathcal{O}(m^2kf(m))$ time. 
	
	The upper envelope of $\m$ piecewise linear, real-valued and continuous functions can be computed in $\mathcal{O}(t \log t)$ time, cf.\ \cite{HERSHBERGER1989169}, where $t$ is the total number of linear pieces in the graphs of the functions $y_e$.
	For our problem, $t$ is in $\mathcal{O}(\m^2 k^{\frac{1}{3}})$ since each of the $\m$ functions $y_e$ has at most $\mathcal{O}(\m k^{\frac{1}{3}})$ many breakpoints, cf.\ \cite{dey1998improved}. 
	The most vital element $e^*$ at a point $\lambda$ is just the element $e^*$ for which $y(\lambda) = y_{e^*}(\lambda)$.
	Hence, the computation of $y$ as the upper envelope of the functions $y_e$ also provides us with the most vital elements.
	
	In summary, we obtain a total running time of $$\mathcal{O}(m^2\log m + m^2kf(m) + \m^2 k^{\frac{1}{3}} \log(\m k)) = \mathcal{O}(m^2(kf(m)+k^{\frac{1}{3}}\log(mk))).$$
\end{proof}

Next, we want to make use on the reduction of the number of potential breakpoints provided by \cref{theorem mk breakpoints}.
To this end, we need to be able to compute this reduced candidate set.
\begin{theorem} \label{theorem mk breakpoints computation}
	We can compute a set of equality points of cardinality $\mathcal{O}(mk)$ that contains all breakpoints of the functions $y_e$ in \O{\m^2(\log m + f(\m)) + \m G(\m)} time.
\end{theorem}
\begin{proof}
	We make use of the proof of \cref{theorem mk breakpoints} to obtain the desired set:
	to this end, we note that, if an equality point $\lambda\coloneqq \lambda(e\to f)$ is a candidate for a breakpoint, then either Case~\ref{itm:matroid-restr-rank-increase} or Case~\ref{itm:matroid-restr-same-rank} occurs.
	In Case~\ref{itm:matroid-restr-same-rank}, the element $f$ added to $\M_e(\lambda)$ is in a circuit with an element $g$ that forms its own component.
	Thus, we can proceed as outlined in \cref{alg:bp-candidates} to get the set of candidates.
	
	We first compute all equality points and sort them in ascending order, which requires \O{m^2\log m} time.
	Next, we pick a point $\lambda$ before the first equality point and initialise the matroids $\M_e(\lambda)$ by computing the elements they contain, a basis for each, and their components.
	The first two parts take \O{m\log m} and \O{m^2f(m)} time, respectively.
	To get the set of candidates, we now just need to iterate over the equality points in the sorted order, checking whether Case~\ref{itm:matroid-restr-rank-increase} or Case~\ref{itm:matroid-restr-same-rank} occurs at each.
	If so, we add it to the list of candidates and otherwise it can safely be skipped.
	Since these cases can occur at most \O{mk} times, the result follows.
	
	To see how these cases can be checked, we start with the former.
	By checking independence of $B^e + f$ we can determine whether the rank of $\M_e$ increases when we transition to $\M_e + f$.
	If it is independent, we are in Case~\ref{itm:matroid-restr-rank-increase} and we add it to the set of candidates.
	We also update the basis, so our variables remain up to date.
	This takes \O{f(m)} time.
	
	Afterwards, we compute the components of $\M_e+f$ time from those of $\M_e$.
	To complete the checking of the equality point, we check the component of~$f$, which is the union of prior components with the set $\Set{f}$.
	If one of the components unified is a singleton, we are in Case~\ref{itm:matroid-restr-same-rank} and we add it to the list of candidates.
	Otherwise, we are in neither of the two cases and are not faced with a breakpoint, letting us disregard the equality point.
	
	The only running time we have not yet analysed are the component updates of the matroids $\M_e(\lambda)$.
	Since these are only $m$ matroids to which we add individual elements at most $m$ times, the running time is bounded by $G(m)$ for each.
\end{proof}

\begin{algorithm}[t]
	\SetKwFunction{FindCandidates}{FindCandidates}
	\KwIn{A matroid \M{} with edge weights $w(e,\lambda)$ and parameter interval~$I$.}
	\KwOut{A set of \O{mk} equality points that contain all breakpoints of the functions $y_e$ and $w$.}
	\Def{\FindCandidates{$\M$, $w(e,\lambda)$, $I$}}{
		Compute all equality points $\lambda(e\to f)$ and sort them ascendingly\;
		Let $\lambda$ be a point before the first equality point\;
		$\M_e\algass \M_e(\lambda)$ for all $e\in E$\;
		Compute a basis $B^e$ of $\M_e$ and the components of $\M_e$ for all $e\in E$\;
		$C\algass \emptyset$\;
		\For{equality point $\lambda\coloneqq\lambda(e\to f)$}
		{
			\uIf{$B^e + f$ is independent}
			{
				$C \algass C \cup \{\lambda\}$\;
				$B^e \algass B^e + f$\;
			}
			Compute components of $\M_e + f$ from those of $\M_e$\;
			$\M_e \algass \M_e + f$\;
			\uIf{the component of $f$ subsumed a singleton}
			{
				$C \algass C \cup \{\lambda\}$\;
			}
		}
	}
	\caption{An algorithm for computing a candidate set of \O{mk} breakpoints.
		\label{alg:bp-candidates}}
\end{algorithm}
\setcounter{algocf}{1}

We can use this procedure to shorten the first step of the algorithm in \cref{theorem upper bound 1 computation} from $\mathcal{O}(m^2kf(m))$ to $\mathcal{O}(mk^2f(m) + \m^2(\log m + f(\m)) + \m G(\m))$ time.
This yields the following result.
\begin{corollary}
	\label{theorem upper bound 1 computation improved}
	The optimal interdiction value function $y$ of the parametric matroid one-interdiction problem can be computed in time
	\begin{displaymath}
		\O{m^2f(m) + mk^2f(m)+m^2k^{\frac{1}{3}}\log(mk)+mG(m)}.
	\end{displaymath}
\end{corollary}

Next, we shall apply an alternative approach to computing $y$.
Instead of computing the functions $y_e$ to then obtain $y$ as their upper envelope, we now determine the behaviour of the functions $y_e$ only between two consecutive (potential) breakpoints.
Then we determine their upper envelope locally on this interval before proceeding to the next one.
This results in the following running time.
\begin{theorem} \label{thm: algorithm breakpoint intervals}
	\Cref{problem matroid} can be solved in time
	\begin{displaymath}
		\O{\m^2(\log m + f(\m)) + \m G(\m) + mk(H(m) + k\log k)}.
	\end{displaymath}
\end{theorem}
\begin{proof}
	The algorithm we describe is illustrated in \cref{alg:bp-intervals}.
	First, we run \cref{alg:bp-candidates} to get a set of candidates for potential breakpoints $\lambda_1,\ldots,\lambda_r$.
	We then define $I_0\coloneqq (-\infty,\lambda_1]$, $I_i\coloneqq [\lambda_i,\lambda_{i+1}]$ for $1\leq i < r$, and $I_r\coloneqq [\lambda_r,\infty)$.
	
	Next, we note that, inside the intervals, the optimal value function $w$ remains linear and corresponds to the same basis.
	Thus, we compute optimal bases $B_i$ for $\M$ for all intervals $I_i$, $i=0,\ldots,r$.
	This only requires $\O{mf(m)}$ for the first one and a single independence test at any equality point $\lambda_i$, for a total of \O{mkf(m)} time.
	
	We now compute $y$ for the intervals $I_i$ individually and combine the results.
	To do so, we first determine $y_e$ on $I_i$ for $e\in B_i$.
	This can be done by taking $B_i$ and computing replacement elements $r_e\coloneqq r_\lambda(e)$ for all $e\in B_i$ at some $\lambda\in I_i$ in \O{H(m)} time.
	By \cref{lemma replacement element}, the bases $B_i - e + r_e$ are optimal at $\lambda$ and thus on all of $I_i$, giving us $y_e(\lambda) = w(B_i - e + r,\lambda)$ here.
	
	Using the $y_e$ and \cref{lemma only elements of B*}, we can compute $y$ on $I_i$ as the upper envelope of the $y_e$ using the algorithm in \cite{HERSHBERGER1989169}, which requires $\O{k\log k}$ time.
	The combined running time thus amounts to
	\begin{align*}
		&\O{\m^2(\log m + f(\m)) + \m G(\m) + mkf(m) + mkH(m) + mk^2\log k} \\
		=\; &\O{\m^2(\log m + f(\m)) + \m G(\m) + mk(H(m) + k\log k)}.
	\end{align*}
\end{proof}
\begin{algorithm}[t]
	\Input{A matroid \M{} with edge weights $w(e,\lambda)$ and parameter interval~$I$.}
	\Output{A representation of the upper envelope of $y$.}
	
	Compute a set of candidate breakpoints $\lambda_1,\ldots,\lambda_r$ using \cref{alg:bp-candidates}\;
	Let $I_0,\ldots,I_{r}$ be the corresponding intervals\;
	Compute optimal bases $B_i$ for $\M$ on $I_i$ for all $i=0,\ldots,r$\;
	\For{$i = 0, \dotsc, r$}{
		\For{$e \in B_i$}{
			Compute the replacement element $r_e$ for $e$\;
			Obtain $y_e$ on $I_i$ by using the optimal basis $B_i - e + r_e$\;
		}
		Compute $y$ on $I_i$\;
	}
	\Return $y$\;
	\caption{An algorithm for computing the optimal interdiction value function~$y$ based on intervals between breakpoints.
		\label{alg:bp-intervals}}
\end{algorithm}

\textbf{Algorithms for the minimum spanning tree interdiction problem.}
We now look at the special case for graphic matroids and see how our running times behave when we can implement the functions $f$, $G$, and $H$.
First we note that an independence test for graphic matroids is a cycle test, which can be done in $\mathcal{O}(\log n)$ time using the dynamic tree data structure of \cite{sleator1981data}.
This yields the following upper bound as a direct consequence of \cref{theorem upper bound 1 computation,theorem mk^2 changepoints} as well as a lower bound using the proof technique from \cref{theorem lower bound} combined with the lower bound of $\Omega(m \log n)$ for the parametric minimum spanning tree problem proved in \cite{eppstein1995geometric}.
\begin{corollary}\label{theorem bounds mst}
	The optimal interdiction value function $y$ of the parametric minimum spanning tree one-interdiction problem has at least $\Omega(m \log n)$ and at most $\mathcal{O}(m n^2)$ many changepoints that can be computed in $\mathcal{O}(m^2n\log n)$ time.
\end{corollary}

The running time in \cref{theorem bounds mst} is still dominated by solving $m$ parametric minimum spanning tree problems.
In order to get an improvement by using \cref{theorem upper bound 1 computation improved}, we need to implement the function $G(m)$.
This means that we need to keep track of the 2-connected components of the graph $H=(V,\emptyset)$ to which we iteratively add edges, a task that can be done in \O{m+n\log n} or $\O{m\alpha(m)}$ time cf.\ \cite{WT92}.
\begin{corollary} \label{thm:improved-alg}
	\Cref{problem mst} can be solved     in $\mathcal{O}(mn^2\log n + m^2n^{\tfrac{1}{3}}\log n)$ time.
\end{corollary}

Finally, to make use of \cref{alg:bp-intervals} and \cref{thm: algorithm breakpoint intervals}, we need to implement $H(m)$ as well.
This can be done in $\mathcal{O}(m\alpha(m))$ time, cf.\ \cite{IWANO}.
Sadly, the resulting running time of $\mathcal{O}(m^2n\alpha(m) + mn^2\log n)$ only improves the one from \cref{theorem upper bound 1 computation}, but not the one in \cref{thm:improved-alg}.

\section{Conclusion}\label{sec:conclusion}
In this article, we have investigated the parametric matroid one-interdiction problem which has not been considered in the literature so far.
Given a matroid where the weight of each element depends on a real parameter from an interval, the goal is to compute for each parameter value a so-called most vital element, that is, an element that, when removed from the matroid, maximizes the weight of the minimum weight basis.

We have proved structural properties of the problem such as a lower bound of $\Omega(m k^{\frac{1}{3}})$ and two upper bounds of $\mathcal{O}(m^2 k^{\frac{1}{3}}\alpha(k))$ and $\mathcal{O}(mk^2)$ on the number of changepoints of the piecewise linear and non-concave optimal interdiction value function.
Here $m$ is the number of elements, $k$ the rank of the matroid, and $\alpha$ is a functional inverse of Ackermann's function.
Using these properties, we have developed algorithms whose running time depends on different matroid operations. 
Finally, we have specified an implementation of these operations for graphical matroids to obtain a polynomial-time algorithm for the parametric minimum spanning tree one-interdiction problem.
	
\section*{Acknowledgments}
	This work was partially supported by the project Ageing Smart funded by the Carl-Zeiss-Stiftung and the DFG grant RU 1524/8-1, project number 508981269.
    Furthermore we would like to thank Dorothee Henke for improvements and useful hints to the literature.
	
%\section*{References}
	
\bibliography{bib_matroid}
	
\end{document}